\documentclass[english]{article}
\usepackage[T1]{fontenc}
\usepackage[latin9]{inputenc}
\usepackage{mathrsfs}
\usepackage{amsmath}
\usepackage{amsthm}
\usepackage{amssymb}
\usepackage{graphicx}
\usepackage{esint}

\makeatletter
\numberwithin{equation}{section}
\numberwithin{figure}{section}

\newtheorem{theorem}{Theorem}[section]
\newtheorem{lemma}[theorem]{Lemma}

\theoremstyle{definition}

\theoremstyle{remark}
\newtheorem*{remark}{Remark}

\newcommand\shorttitle{The fine structure of a simple Markov process}
\newcommand\authors{Andreas Anckar and Göran Högnäs}

\usepackage{fancyhdr}
\@ifundefined{showcaptionsetup}{}{%
 \PassOptionsToPackage{caption=false}{subfig}}
\usepackage{subfig}
\makeatother

\usepackage{babel}
\begin{document}

\title{The fine structure of the stationary distribution for a simple Markov
process\thanks{This manuscript version is made available under the CC-BY-NC-ND 4.0
license: https://creativecommons.org/licenses/by-nc-nd/4.0/ }}

\author{Andreas Anckar\\
Göran Högnäs\\
Åbo Akademi University\\
Fänriksgatan 3, 20500 Turku, Finland}

\date{29.6.2015}
\maketitle
\begin{abstract}
We study the fractal properties of the stationary distrubtion $\pi$
for a simple Markov process on $\mathbb{R}$. We will give bounds
for the Hausdorff dimension of $\pi$, and lower bounds for the multifractal
spectrum of $\pi$. Additionally, we will provide a method for numerically
estimating these bounds.
\end{abstract}

\section{Introduction}

For real numbers $\alpha>1,\beta>0$, we define a Markov process by
\begin{equation}
X_{n+1}=\begin{cases}
X_{n}+\beta & \text{with probability }p\\
\alpha^{-1}X_{n} & \text{with probability }1-p.
\end{cases}\label{eq:markov}
\end{equation}
\begin{figure}
\caption{\label{fig:Xn}Plot of $X_{n}$ when $\alpha=2$, $\beta=1$ and $p=1/3$}

\subfloat[Histogram]{

\includegraphics[width=0.28\paperwidth]{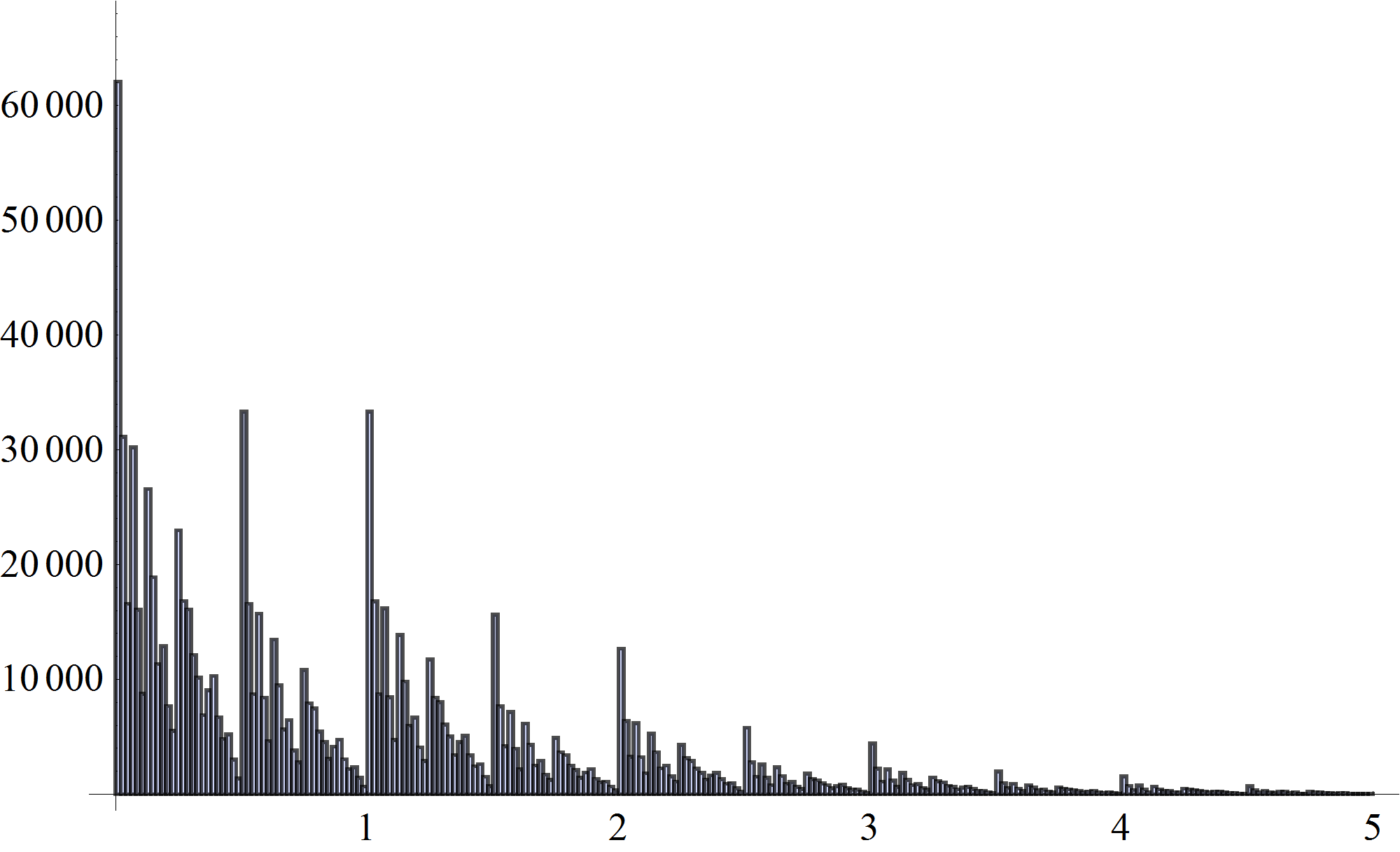}}\subfloat[Empirical mass distribution]{

\includegraphics[width=0.28\paperwidth]{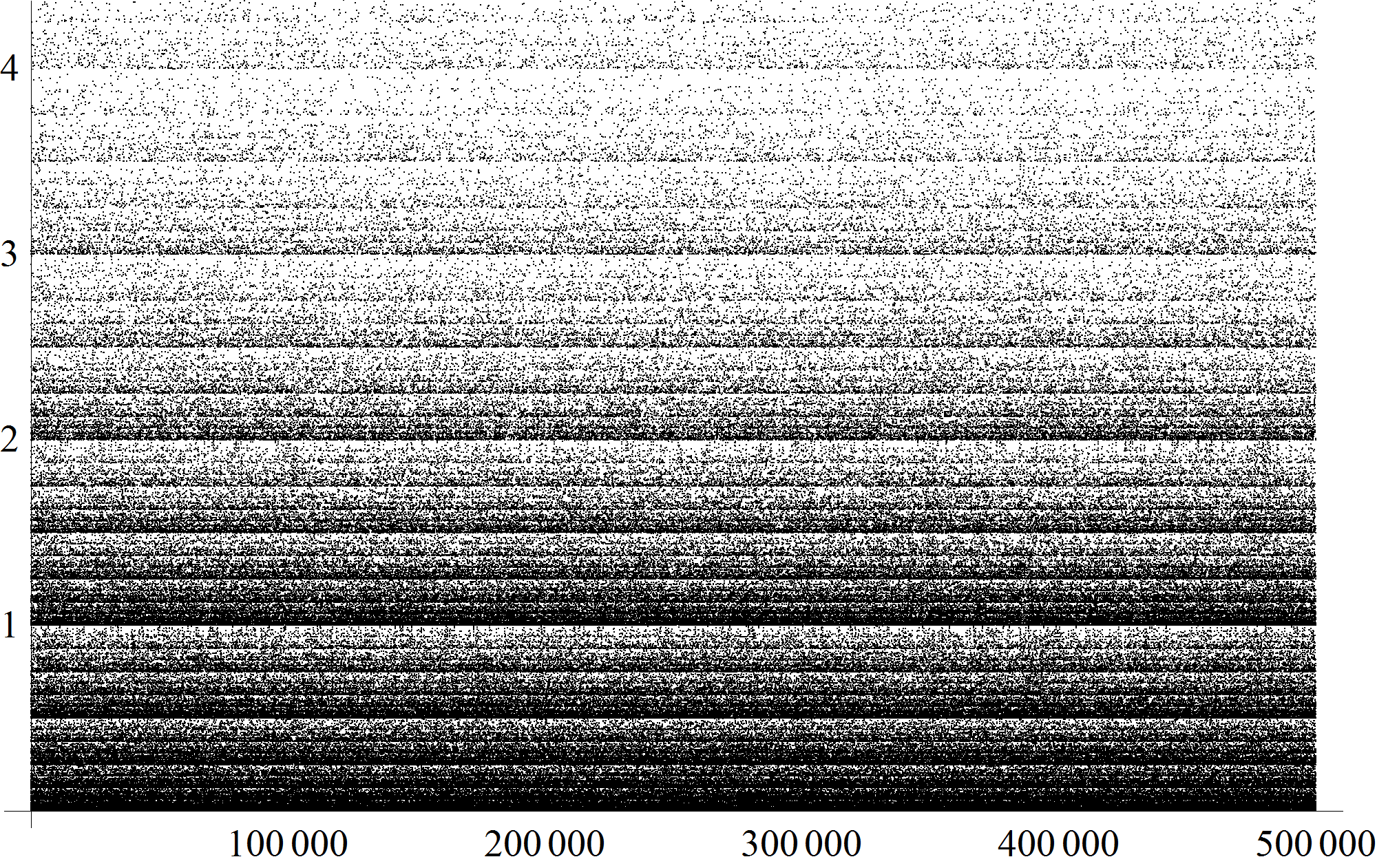}}
\end{figure}
We will denote the stationary distribution of $X_{n}$ by $\pi$.
As figure \ref{fig:Xn} shows, this distribution exhibits typical
fractal patterns. In order to acquire a solid framework in which we
can study the fine structure (ie. Hausdorff dimension and multifractal
spectrum) of $\pi$, we will reformulate the process $X_{n}$ in the
context of iterated function systems. 

A (probabilistic) iterated function system (IFS) is a set $\mathbb{X}\subset\mathbb{R}^{d}$
associated with a family of maps $\mathscr{W}=\left\{ w_{i}\right\} _{i=1}^{N}$,
$w_{i}:\mathbb{X}\rightarrow\mathbb{X}$, where the maps are chosen
independently according to a probability vector $\mathbf{p}=\left\{ p_{i}\right\} _{i\in M}$,
where $p_{i}>0$ for all $i=1,\ldots,N$ and $\sum_{i=1}^{N}p_{i}=1$.
The maps are all Lipschitz, ie. there exist positive constants $\gamma_{i}$
such that $\left|w_{i}(x)-w_{i}(y)\right|\leq\gamma_{i}\left|x-y\right|$
for all $x,y\in\mathbb{X}$ and $i=1,\ldots,N$. If $\gamma_{i}<1$
for all $i$ the IFS is said to be strictly contracting, but a weaker
condition is that $\sum_{i=1}^{N}p_{i}\log\gamma_{i}<0$, in which
case the IFS is said to fulfill average contractivity. In either case
there exists a unique probability measure on $\mathbb{X}$ satisfying
\[
\mu=\sum_{i=1}^{N}p_{i}\mu\circ w_{i}^{-1},
\]
which is called the invariant measure of the IFS (see \cite{Freedman1999}
for a proof). In other terms, if we put $\Sigma=\{1,2,\ldots,N\}$
and let $\mathbb{P}$ be the infinite-fold product probability measure
$\mathbf{p}\times\mathbf{p}\times\cdots$ on $\Sigma^{\infty}$, the
limit 
\begin{equation}
\nu(\mathbf{i})=\lim_{n\rightarrow\infty}w_{i_{1}}\circ w_{i_{2}}\circ\cdots\circ w_{i_{n}}\left(x_{0}\right)\label{eq:nu}
\end{equation}
exists for $\mathbb{P}$-almost every sequence $\mathbf{i}=\left(i_{1},i_{2},\ldots\right)\in\Sigma^{\infty}$,
and does not depend on $x_{0}\in\mathbb{X}$. The mapping $\nu:\Sigma^{\infty}\rightarrow\mathbb{X}$
is thus well defined almost everywhere on $\mathbb{P}$ and $\mu$
can be written as
\[
\mu=\mathbb{P}\circ\nu^{-1}.
\]
Now, let $\Xi(\alpha,\beta,p)$ be the family of IFS's on $\mathbb{R}$
of the form 
\begin{eqnarray*}
w_{1}(x) & = & x+\beta\\
w_{2}(x) & = & \alpha^{-1}x
\end{eqnarray*}
with probability vector $\mathbf{p}=(p,1-p)$, and $\alpha>1,\beta>0$.
The IFS isn't strictly contracting, since $\gamma_{1}=1$, but $\sum p_{i}\log\gamma_{i}=-(1-p)\log\alpha<0$
shows that average contractivity holds. Thus the unique invariant
measure $\mu$ exists and satisfies the recursion relation
\begin{equation}
\mu(A)=p\mu(A-\beta)+(1-p)\mu(\alpha A)\label{eq:recursion}
\end{equation}
for any measurable $A\subset\mathbb{R}$. By writing $X_{n+1}=w_{i_{n+1}}\left(X_{n}\right)$,
where $i_{n}$ is drawn randomly according to $\mathbb{P}$, we see
that the above IFS represents the same random process as the initial
Markov process (\ref{eq:markov}), and $\mu$ is indeed equal to $\pi$.
We will henceforth refer to this measure by $\pi$.

The following notions related to fractal geometry will largely follow
the same definitions as in eg. \cite{Falconer1997}. The notation
$\dim_{H}$ will be used for the Hausdorff dimension of a set. For
any Borel probability measure $\mu$ on $\mathbb{R}$, the lower local
dimension of $\mu$ at $x\in\mathbb{R}$ is defined by 
\begin{equation}
\underline{\dim}\mu(x)=\liminf_{r\rightarrow0}\frac{\log\mu\left(B(x,r)\right)}{\log r}\label{eq:lolocaldim}
\end{equation}
The upper and lower Hausdorff dimensions of $\mu$ are now given by

\begin{eqnarray}
\dim_{H}^{*}\mu & = & \inf\left\{ s:\,\underline{\dim}\mu(x)\leq s\text{ for }\mu\text{-almost all }x\right\} \label{eq:updim}\\
\dim_{H}\mu & = & \sup\left\{ s:\,\underline{\dim}\mu(x)\geq s\text{ for }\mu\text{-almost all }x\right\} ,\label{eq:lodim}
\end{eqnarray}
respectively. Note that $\dim_{H}\mu\leq\dim_{H}^{\star}\mu$. Now
let 
\[
\overline{E}_{t}=\left\{ x\in\mathbb{R}:\,\dim\pi(x)\leq t\right\} ,\;\underline{E}_{t}=\left\{ x\in\mathbb{R}:\,\dim\pi(x)\geq t\right\} ,
\]
and similarly $\overline{f}_{H}(t)=\dim_{H}\overline{E}_{t}$, $\underline{f}_{H}(t)=\dim_{H}\underline{E}_{t}$.
We call the functions $\overline{f}_{H}(t)$ and $\underline{f}_{H}(t)$
the upper and lower multifractal spectrum of $\mu$, respectively.

The Hausdorff dimension of invariant measures of IFS's have been studied
extensively in the last decades. With light conditions on the maps
in $\mathscr{W}$ and only assuming average contractivity, in general
only upper bounds for the Hausdorff dimension of $\mu$ are known
(see eg. \cite{Jordan2008}). The usual way of finding lower bounds
is by trying to limit the overlap of the maps. This is most commonly
done by assuming the open set condition (OSC), which is fulfilled
if there exists an open set $\mathcal{O}\subset\mathbb{X}$ such that
$w_{i}(\mathcal{O})\subset\mathcal{O}$ and $w_{i}(\mathcal{O})\cap w_{j}(\mathcal{O})=\emptyset$
for all $i\neq j$. If this condition fails, there are a few weaker
assumptions that have yielded results (see \cite{Lau2008} for a survey).
In the simple case where the measure has compact support and the maps
in $\mathscr{W}$ are strictly contracting similitudes satisfying
the OSC, the geometry is fully understood. The IFS we study here is
of interest because it does not satisfy the OSC, nor any of the other
overlap conditions. The only known result applicable to our process
is
\[
\dim_{H}^{\star}\pi\leq\frac{p\log p+(1-p)\log(1-p)}{(1-p)\log\alpha^{-1}}.
\]
In theorem \ref{main} we present a strictly smaller upper bound,
and a lower bound as well. We also obtain lower bounds for the multifractal
spectrum.

For any positive integer $b$ and $x\in\mathbb{R}$, let $\delta_{i}^{b}(x)$
denote the $i$:th digit of a base-$b$ expansion of $x/b^{\left\lfloor \log_{b}x\right\rfloor +1}=0.\delta_{1}^{b}\delta_{2}^{b}\ldots\in[0,1]$.
This representation is unique except for points whose expansion ends
in an infinite sequence of $0$'s, since such numbers may also be
written as an expansion ending in an infinite sequence of $(b-1)$'s.
We will ensure uniqueness of $\delta_{i}^{b}(x)$ by always choosing
the former representation in such cases. Write $\tau_{k}^{b}(x,n)$
for the number of occurrences of the digit $k$ in the $n$ first
digits of the base-$b$ expansion of $x$. Whenever it exists, we
denote $\tau_{k}^{b}(x)=\lim_{n\rightarrow\infty}\frac{1}{n}\tau_{k}^{b}(x,n)$.
For any vector $\left(q_{0},q_{1},\ldots,q_{b-1}\right)$ of non-negative
real numbers with $\sum_{i=0}^{b-1}q_{i}=1$ we define 
\[
F_{b}\left(q_{0},q_{1},\ldots,q_{b-1}\right)=\left\{ x\in\mathbb{R}:\,\tau_{k}^{b}(x)=q_{k},\:k=0,1,\ldots,b-1\right\} .
\]
Furthermore, let 
\[
S_{b,t}=\left\{ x\in\mathbb{R}:\,\lim_{n\rightarrow\infty}\frac{1}{n}\sum_{i=1}^{n}\delta_{i}^{b}(x)=t\right\} =\left\{ x\in\mathbb{R}:\,\sum_{i=0}^{b-1}i\tau_{i}^{b}(x)=t\right\} .
\]
By a classical result of Eggleston (\cite{Eggleston1949}), 
\begin{equation}
\dim_{H}\left(F_{b}\left(q_{0},q_{1},\ldots,q_{b-1}\right)\right)=-\sum_{i=\text{1}}^{b-1}q_{i}\log q_{i}/\log b.\label{eq:egg}
\end{equation}
The Hausdorff dimension of the set $S_{b,t}$ is known to be (see
\cite{Barreira2002}, corollary 15)
\begin{equation}
\dim_{H}\left(S_{b,t}\right)=\frac{\log\min\left\{ \sum_{i=0}^{b-1}r^{i-t}:\,r>0,\,\sum_{i=0}^{b-1}(i-t)r^{i}=0\right\} }{\log b}.\label{eq:barrdim}
\end{equation}
For any $\alpha>1$ and $0\leq p\leq1$, we set 
\[
\overline{d}(x)=\frac{\log(1-p)+x\log p}{\log\alpha^{-1}},\quad\underline{d}(x)=\overline{d}(x)-\frac{\log\left(1-p^{\alpha-1}\right)}{\log\alpha^{-1}}
\]
and 
\[
d^{\star}(x)=\min\left\{ \max\left\{ \overline{d}^{-1}(x),0\right\} ,1\right\} ,\;d_{\star}(x)=\min\left\{ \max\left\{ \underline{d}^{-1}(x),0\right\} ,1\right\} 
\]
for all $0\leq x\leq1$. We are now ready to state the main result:

\begin{theorem}\label{main}For an IFS in $\Xi(\alpha,\beta,p)$ where
$\alpha\geq2$ and $\beta\geq1$ are integers we have
\begin{equation}
\dim_{H}^{\star}\pi\leq\frac{\sum_{i=0}^{\alpha-1}\xi_{i}\log\xi_{i}}{\log\alpha^{-1}},\label{eq:mainup}
\end{equation}
where $\xi_{k}=\sum_{m=0}^{\infty}\pi\left[m\alpha+k,m\alpha+k+1\right]$.
Moreover, if $p\leq1/2$ and $\beta=\alpha^{t}$ for some $t=0,1,2,\ldots$,
then 
\begin{equation}
\dim_{H}\pi\geq\underline{d}\left(\sum_{i=0}^{\alpha-1}i\xi_{i}\right)\label{eq:mainlo}
\end{equation}
and 
\begin{equation}
\overline{f}_{H}(t)\geq\dim_{H}\left(S_{\alpha,d^{\star}(t)}\right),\;\underline{f}_{H}(t)\geq\dim_{H}\left(S_{\alpha,d_{\star}(t)}\right).\label{eq:mainmf}
\end{equation}
\end{theorem}

\section{Statement of results}

In this section, if not otherwise stated, we will assume that $\pi$
is the invariant measure of an IFS in $\Xi(\alpha,\beta,p)$, where
$p\leq1/2$ and $\alpha\geq2$ and $\beta\geq1$ are integers.

\begin{lemma}\label{hf}For any non-negative $x$ we have
\[
\pi[x,\alpha x]=\frac{p}{1-p}\pi[x-\beta,x].
\]
\end{lemma}\begin{proof}

\begin{eqnarray*}
\pi[x,\alpha x] & = & \pi[0,\alpha x]-\pi[0,x]\\
 & = & \pi[0,\alpha x]-(1-p)\pi[0,\alpha x]-p\pi[0,x-\beta]\\
 & = & p\left(\pi[x-\beta,x]+\pi[x,\alpha x]\right)
\end{eqnarray*}
\end{proof}

\begin{lemma}\label{inteq1}Write $M_{0}=\pi\left[0,\beta\right]$.
There exists a constant $K>1$ such that 
\[
M_{0}p^{n}\leq\pi[n\beta,(n+1)\beta]\leq M_{0}Kp^{n}
\]
for all integers $n\geq0$.\end{lemma}\begin{proof}Assume that $n\geq1$
(If $n=0$, the proposition holds for any $K\geq1$). The lower bound
follows immediately from (\ref{eq:recursion}) since
\[
\pi[n\beta,(n+1)\beta]\geq p\pi[(n-1)\beta,n\beta]\geq p^{2}\pi[(n-2)\beta,(n-1)\beta]\geq\ldots\geq p^{n}M_{0}
\]
For the upper bound, we first use lemma \ref{hf} and the facts that
$\alpha\geq2$ and $p\leq1/2$ to note that 
\[
\pi[n\beta,(n+1)\beta]\leq\pi[n\beta,n\alpha\beta]=\frac{p}{1-p}\pi[(n-1)\beta,n\beta]\leq\pi[(n-1)\beta,n\beta].
\]
This implies that $\pi\left[m\beta,n\beta\right]\geq(n-m)\pi\left[(n-1)\beta,n\beta\right]$
for any integers $n>m\geq0$. Thus 
\[
\frac{p}{1-p}\pi[(n-1)\beta,n\beta]=\pi[n\beta,n\alpha\beta]\geq n(\alpha-1)\pi[(n\alpha-1)\beta,n\alpha\beta].
\]
The above and lemma \ref{hf} give
\begin{eqnarray}
\pi[n\beta,(n+1)\beta] & \leq & p\pi[(n-1)\beta,n\beta]+(1-p)\pi[n\alpha\beta,n\alpha^{2}\beta]\nonumber \\
 & = & p\left(\pi[(n-1)\beta,n\beta]+p\pi[(n\alpha-1)\beta,n\alpha\beta]\right)\label{eq:upbou1}\\
 & \leq & p\pi[(n-1)\beta,n\beta]\left(1+\frac{p}{1-p}\cdot\frac{1}{n(\alpha-1)}\right).\label{eq:upbou2}
\end{eqnarray}
For any integers $n\geq m\geq1$, let $P(x,m,n)=\prod_{k=n-m+1}^{n}\left(1+x/k\right)$.
By writing $a=\frac{p}{(1-p)(\alpha-1)}$ and repeating (\ref{eq:upbou2})
we get 
\begin{eqnarray}
\pi[n\beta,(n+1)\beta] & \leq & p^{m}\pi[(n-m)\beta,(n-m+1)\beta]P(a,m,n)\label{eq:upbou3}
\end{eqnarray}
Now, apply (\ref{eq:upbou3}) to the second term in (\ref{eq:upbou1})
to see that 
\begin{eqnarray*}
\pi[n\beta,(n+1)\beta] & \leq & p\pi[(n-1)\beta,n\beta]\left(1+p^{n(\alpha-1)}P\left(a,n(\alpha-1),n\alpha-1\right)\right)
\end{eqnarray*}
A standard result is that $P(x,n,n)=\Gamma(x+n)/\Gamma(x)$ (where
$\Gamma(x)$ denotes the gamma function), which implies $P(1,n,n)=n+1$.
Since $a\leq1$ we have $P\left(a,n(\alpha-1),n\alpha-1\right)<P(1,n,n\alpha-1)=n\alpha$.
Thus we arrive at
\begin{eqnarray*}
\pi[n\beta,(n+1)\beta] & < & p^{n}M_{0}\prod_{k=1}^{n}\left(1+k\alpha p^{k(\alpha-1)}\right)<p^{n}M_{0}\prod_{k=1}^{\infty}\left(1+k\alpha p^{k(\alpha-1)}\right).
\end{eqnarray*}
The infinite product above converges if and only if the series $\sum_{k=1}^{\infty}k\alpha p^{k(\alpha-1)}$
converges, which is clearly the case.

\end{proof}

\begin{lemma}\label{intexp}For any integers $n$ and $k$, 
\[
\pi\left[\frac{n\beta}{\alpha^{k}},\frac{(n+1)\beta}{\alpha^{k}}\right]=(1-p)\sum_{k=0}^{\left\lfloor n\alpha^{-k}\right\rfloor }p^{k}\pi\left[\frac{n\beta}{\alpha^{k-1}}-k\alpha\beta,\frac{(n+1)\beta}{\alpha^{k-1}}-k\alpha\beta\right].
\]
\end{lemma}\begin{proof}The formula is straightforward to obtain
using (\ref{eq:recursion}). We have
\begin{eqnarray*}
 &  & \pi\left[\frac{n\beta}{\alpha^{k}},\frac{(n+1)\beta}{\alpha^{k}}\right]\\
 &  & =p\pi\left[\frac{n\beta}{\alpha^{k}}-\beta,\frac{(n+1)\beta}{\alpha^{k}}-\beta\right]+(1-p)\pi\left[\frac{n\beta}{\alpha^{k-1}},\frac{(n+1)\beta}{\alpha^{k-1}}\right].
\end{eqnarray*}
The first term above can be written as
\begin{eqnarray*}
p\pi\left[\frac{n\beta}{\alpha^{k}}-\beta,\frac{(n+1)\beta}{\alpha^{k}}-\beta\right] & = & p^{2}\pi\left[\frac{n\beta}{\alpha^{k}}-2\beta,\frac{(n+1)\beta}{\alpha^{k}}-2\beta\right]+\\
 &  & p(1-p)\pi\left[\frac{n\beta}{\alpha^{k-1}}-\alpha\beta,\frac{(n+1)\beta}{\alpha^{k-1}}-\alpha\beta\right].
\end{eqnarray*}
By repeatedly using (\ref{eq:recursion}) on the first terms, we generally
have
\begin{eqnarray*}
 &  & p^{j}\pi\left[\frac{n\beta}{\alpha^{k}}-j\beta,\frac{(n+1)\beta}{\alpha^{k}}-j\beta\right]=\\
 &  & p^{j+1}\pi\left[\frac{n\beta}{\alpha^{k}}-(j+1)\beta,\frac{(n+1)\beta}{\alpha^{k}}-(j+1)\beta\right]\\
 &  & +p^{j}(1-p)\pi\left[\frac{n\beta}{\alpha^{k-1}}-j\alpha\beta,\frac{(n+1)\beta}{\alpha^{k-1}}-j\alpha\beta\right].
\end{eqnarray*}
Combining everything yields
\begin{eqnarray*}
\pi\left[\frac{n\beta}{\alpha^{k}},\frac{(n+1)\beta}{\alpha^{k}}\right] & = & p^{\left\lfloor n\alpha^{-k}\right\rfloor }\pi\left[\beta\left(\frac{n}{\alpha^{k}}-\left\lfloor \frac{n}{\alpha^{k}}\right\rfloor \right),\beta\left(\frac{n+1}{\alpha^{k}}-\left\lfloor \frac{n}{\alpha^{k}}\right\rfloor \right)\right]+\\
 &  & (1-p)\sum_{j=0}^{\left\lfloor n\alpha^{-k}\right\rfloor -1}p^{j}\pi\left[\frac{n\beta}{\alpha^{k-1}}-j\alpha\beta,\frac{(n+1)\beta}{\alpha^{k-1}}-j\alpha\beta\right].
\end{eqnarray*}
Since $\alpha$ is an integer, we have $\frac{n}{\alpha^{k}}-\left\lfloor \frac{n}{\alpha^{k}}\right\rfloor \geq0$
and $\frac{n+1}{\alpha^{k}}-\left\lfloor \frac{n}{\alpha^{k}}\right\rfloor \leq1$
for all $n,k$. Thus 
\begin{eqnarray*}
 &  & \pi\left[\beta\left(\frac{n}{\alpha^{k}}-\left\lfloor \frac{n}{\alpha^{k}}\right\rfloor \right),\beta\left(\frac{n+1}{\alpha^{k}}-\left\lfloor \frac{n}{\alpha^{k}}\right\rfloor \right)\right]\\
 &  & =(1-p)\pi\left[\frac{n\beta}{\alpha^{k-1}}-\left\lfloor \frac{n}{\alpha^{k}}\right\rfloor \alpha\beta,\frac{(n+1)\beta}{\alpha^{k-1}}-\left\lfloor \frac{n}{\alpha^{k}}\right\rfloor \alpha\beta\right]
\end{eqnarray*}
whereby the proposition follows.

\end{proof}

\begin{lemma}\label{inteq2}For all integers $n,k\geq0$, define
\[
g(n,k)=n\alpha+(1-\alpha)\sum_{i=0}^{k}\left\lfloor n\alpha^{-i}\right\rfloor .
\]
Let $n\geq0$ be arbitrary. Then, for all $k\geq0$,
\[
M_{0}(1-p)^{k}p^{g(n,k)}\leq\pi\left[\frac{n\beta}{\alpha^{k}},\frac{(n+1)\beta}{\alpha^{k}}\right]\leq M_{0}K\left(\frac{1-p}{1-p^{\alpha-1}}\right)^{k}p^{g(n,k)}.
\]
\end{lemma}\begin{proof}The proposition holds for $k=0$ by lemma
\ref{inteq1}. Assume that it holds for $k=t-1$, for some $t>1$.
Then, by lemma \ref{intexp},
\begin{eqnarray*}
\pi\left[\frac{n\beta}{\alpha^{t}},\frac{(n+1)\beta}{\alpha^{t}}\right] & = & (1-p)\sum_{j=0}^{\left\lfloor n\alpha^{-t}\right\rfloor }p^{j}\pi\left[\frac{\left(n-j\alpha^{t}\right)\beta}{\alpha^{t-1}},\frac{\left(n+1-j\alpha^{t}\right)\beta}{\alpha^{t-1}}\right]\\
 & \leq & M_{0}K\frac{(1-p)^{t}}{\left(1-p^{\alpha-1}\right)^{t-1}}\sum_{j=0}^{\left\lfloor n\alpha^{-t}\right\rfloor }p^{j}\cdot p^{g\left(n-j\alpha^{t},t-1\right)}.
\end{eqnarray*}
Notice that 
\begin{eqnarray*}
g\left(n-j\alpha^{t},t-1\right) & = & \left(n-j\alpha^{t}\right)\alpha+(1-\alpha)\sum_{m=0}^{t-1}\left\lfloor \frac{n-j\alpha^{t}}{\alpha^{m}}\right\rfloor \\
 & = & g(n,t-1)+j\left(-\alpha^{t+1}-(1-\alpha)\sum_{m=0}^{t-1}\alpha^{t-m}\right)\\
 & = & g(n,t-1)-j\alpha.
\end{eqnarray*}
Thus
\[
\pi\left[\frac{n\beta}{\alpha^{t}},\frac{(n+1)\beta}{\alpha^{t}}\right]\leq M_{0}K\frac{(1-p)^{t}}{\left(1-p^{\alpha-1}\right)^{t-1}}p^{g(n,t-1)}\sum_{j=0}^{\left\lfloor n\alpha^{-t}\right\rfloor }p^{j(1-\alpha)}
\]
Now, since
\begin{equation}
p^{(1-\alpha)\left\lfloor n\alpha^{-t}\right\rfloor }\leq\sum_{j=0}^{\left\lfloor n\alpha^{-t}\right\rfloor }p^{j(1-\alpha)}\leq\frac{p^{(1-\alpha)\left\lfloor n\alpha^{-t}\right\rfloor }}{1-p^{\alpha-1}},\label{eq:fine_sum_bou}
\end{equation}
and $g(n,t-1)+(1-\alpha)\left\lfloor n\alpha^{-t}\right\rfloor =g(n,t)$,
we have
\[
\pi\left[\frac{n\beta}{\alpha^{t}},\frac{(n+1)\beta}{\alpha^{t}}\right]\leq M_{0}K\left(\frac{1-p}{1-p^{\alpha-1}}\right)^{t}p^{g(n,t)}
\]
For the lower bound, we use a practically identical calculation and
the lower bound in (\ref{eq:fine_sum_bou}) to obtain 
\begin{eqnarray*}
\pi\left[\frac{n\beta}{\alpha^{t}},\frac{(n+1)\beta}{\alpha^{t}}\right] & \geq & M_{0}(1-p)^{t}p^{g(n,t-1)}\sum_{j=0}^{\left\lfloor n\alpha^{-t}\right\rfloor }p^{j(1-\alpha)}\\
 & \geq & M_{0}(1-p)^{t}p^{g(n,t)}.
\end{eqnarray*}
\end{proof}\begin{lemma}\label{dimension}Let $x\in D_{\alpha}$.
For any integer $b\geq2$, define
\[
\underline{\sigma}^{b}(x)=\liminf_{n\rightarrow\infty}\frac{1}{n}\sum_{i=1}^{n}\delta_{i}^{b}(x).
\]
Then
\begin{eqnarray*}
\underline{d}\left(\underline{\sigma}^{\alpha}(x/\beta)\right) & \leq & \underline{\dim}\pi(x)\leq\overline{d}\left(\underline{\sigma}^{\alpha}(x/\beta)\right).
\end{eqnarray*}
\end{lemma}\begin{proof}

First, we remark that for any integer $n\geq0$ the quantity $g(n,k)$
is related to the sum of the digits in the base-$\alpha$ expansion
of $n$. Define $L(x)=\left\lfloor \log_{\alpha}x\right\rfloor +1$,
then 
\begin{eqnarray*}
g(n,k) & = & \sum_{j=0}^{k}\left\lfloor n\alpha^{-j}\right\rfloor -\alpha\left(\sum_{j=0}^{k}\left\lfloor n\alpha^{-j}\right\rfloor -n\right)\\
 & = & \sum_{j=0}^{k}\left\lfloor n\alpha^{-j}\right\rfloor -\alpha\sum_{j=0}^{k-1}\left\lfloor n\alpha^{-j-1}\right\rfloor \\
 & = & \left\lfloor n\alpha^{-k}\right\rfloor +\sum_{i=0}^{k-1}\delta_{L(n)-i}^{\alpha}(n),
\end{eqnarray*}
since $\left\lfloor n\alpha^{-j}\right\rfloor -\alpha\left\lfloor n\alpha^{-j}\right\rfloor =\delta_{L(n)-j}^{\alpha}(n)$.
Now, fix $x\in\mathbb{R}$ and take $\left\{ x_{k}\right\} _{k=0}^{\infty}$
to be the unique sequence of integers satisfying 
\begin{equation}
x\in\left[\frac{x_{k}\beta}{\alpha^{k}},\frac{\left(x_{k}+1\right)\beta}{\alpha^{k}}\right)\label{eq:dim_xprox}
\end{equation}
for every $k\geq0$. Additionally, fix $r\in\left(0,\beta\alpha^{-1}\right]$
and put $k=\max\left\{ k\in\mathbb{N}:\,r\leq\beta\alpha^{-k}\right\} $.
Then $r>\beta\alpha^{-k-1}$ and
\[
\frac{\log\pi\left(B(x,r)\right)}{\log r}\geq\frac{\log\pi\left[\left(x_{k}-1\right)\beta\alpha^{-k},\left(x_{k}+2\right)\beta\alpha^{-k}\right]}{\log\beta\alpha^{-k-1}}.
\]
Define $\overline{x}_{k}$ as the integer in $\left\{ x_{k}-1,x_{k},x_{k}+1\right\} $
for which $\pi\left[\overline{x}_{k}\beta\alpha^{-k},\left(\overline{x}_{k}+1\right)\beta\alpha^{-k}\right]$
is maximized. Then 
\begin{eqnarray*}
\frac{\log\pi\left(B(x,r)\right)}{\log r} & \geq & \frac{\log3\pi\left[\overline{x}_{k}\beta\alpha^{-k},\left(\overline{x}_{k}+1\right)\beta\alpha^{-k}\right]}{\log\beta\alpha^{-k-1}}\\
 & \geq & \frac{\log3M_{0}K+\log\left(\frac{1-p}{1-p^{\alpha-1}}\right)^{k}p^{g\left(\overline{x}_{k},k\right)}}{(k+1)\log\alpha^{-1}+\log\beta},
\end{eqnarray*}
where we applied lemma (\ref{inteq2}) in the second step. Now set
$y_{k}=\overline{x}_{k}\beta\alpha^{-k}$ and $N=\min\left\{ n:\,\beta<\alpha^{n}\right\} $.
By (\ref{eq:dim_xprox}), 
\begin{equation}
\left|x-y_{k}\right|<2\beta\alpha^{-k}\leq\alpha^{N+1-k},\label{eq:dim_dec}
\end{equation}
for $k\geq0$, implying $\delta_{i}\left(x\right)=\delta_{i}\left(y_{k}\right)$
for $1\leq i\leq k^{\prime}$ where $k^{\prime}=k-N-1$. Thus
\[
\sum_{i=1}^{k^{\prime}}\delta_{i}\left(\overline{x}_{k}\right)=\sum_{i=1}^{k^{\prime}}\delta_{i}\left(\frac{\alpha^{k}}{\beta}y_{k}\right)=\sum_{i=1}^{k^{\prime}}\delta_{i}\left(\frac{\alpha^{k}}{\beta}x\right)
\]
for all $k\geq0$, giving 
\begin{eqnarray}
\frac{g\left(\overline{n}_{k},k^{\prime}\right)}{k} & = & \frac{1}{k}\left(\left\lfloor \overline{x}_{k}\alpha^{-k^{\prime}}\right\rfloor +\sum_{i=1}^{k^{\prime}}\delta_{i}\left(\frac{\alpha^{k}}{\beta}x\right)\right).\label{eq:dim_glim}
\end{eqnarray}
Multiplying a number by $\alpha^{k}$ does not affect its digits,
so as $r\rightarrow0$,
\begin{eqnarray*}
\liminf_{r\rightarrow0}\frac{\log\pi\left(B(x,r)\right)}{\log r} & \geq & \liminf_{k\rightarrow\infty}\frac{\log3M_{0}K+k\log\left(\frac{1-p}{1-p^{\alpha-1}}\right)+g\left(\overline{x}_{k},k\right)\log p}{(k+1)\log\alpha^{-1}+\log\beta}\\
 & = & \frac{\log(1-p)-\log\left(1-p^{\alpha-1}\right)+\underline{\sigma}\left(x/\beta\right)\log p}{\log\alpha^{-1}},
\end{eqnarray*}
since $k^{\prime}/k\rightarrow1$. For the upper bound, fix $r$ and
$k$ as before, then
\begin{eqnarray*}
\frac{\log\pi\left(B(x,r)\right)}{\log r} & \leq & \frac{\log\pi\left[x_{k+1}\beta\alpha^{-k-1},\left(x_{k+1}+1\right)\beta\alpha^{-k-1}\right]}{\log\beta\alpha^{-k}}\\
 & \leq & \frac{\log\left(M_{0}(1-p)^{k+1}p^{g\left(x_{k+1},k+1\right)}\right)}{k\log\alpha^{-1}+\log\beta}.
\end{eqnarray*}
As $r$ decreases and $k$ increases, $\left|x-x_{k+1}\beta\alpha^{-k}\right|<\beta\alpha^{-k}$
so (\ref{eq:dim_glim}) holds, whereby 
\begin{eqnarray*}
 &  & \liminf_{r\rightarrow0}\frac{\log\pi\left(B(x,r)\right)}{\log r}\\
 &  & \leq\liminf_{k\rightarrow\infty}\frac{\log M_{0}+(k+1)\log(1-p)+g\left(x_{k+1},k+1\right)\log p}{k\log\alpha^{-1}+\log\beta}\\
 &  & =\frac{\log(1-p)+\underline{\sigma}\left(x/\beta\right)\log p}{\log\alpha^{-1}}.
\end{eqnarray*}
\end{proof}

\begin{lemma}\label{erg}For $\pi$-almost every $x$ we have
\[
\tau_{k}^{\alpha}(x)=\xi_{k}
\]
for $k=0,1,\ldots,\alpha-1$.\end{lemma}\begin{proof}Let $X_{n}$
be as in (\ref{eq:markov}), and write $X_{n}(\mathbf{i})=w_{i_{n}}\circ w_{i_{n-1}}\cdots\circ w_{i_{1}}\left(X_{0}\right)$
for $\mathbf{i}\in\Sigma^{\infty}$. Define $n_{\alpha}\left(X_{n}\right)$
as the number of digits in the base-$\alpha$ expansion of $X_{n}-\left\lfloor X_{n}\right\rfloor $,
and $n_{\alpha}^{\prime}\left(X_{n}\right)$ as the number of digits
in $\left\lfloor X_{n}\right\rfloor $. The number $n_{\alpha}\left(X_{n}\right)$
will equal the number of times the map $w_{2}$ is chosen, so for
$\mathbb{P}$-almost every $\mathbf{i}$, 
\begin{equation}
\lim_{n\rightarrow\infty}\frac{n_{\alpha}\left(X_{n}(\mathbf{i})\right)}{n}=1-p,\label{eq:lln}
\end{equation}
by the law of large numbers. On the other hand, $\left\lfloor X_{n}(\mathbf{i})\right\rfloor $
is at most equal to the number of times $w_{1}$ is chosen, so 
\[
\limsup_{n\rightarrow\infty}\frac{\left\lfloor X_{n}(\mathbf{i})\right\rfloor }{n}\leq p,
\]
$\mathbb{P}$-a.e. It follows that
\[
\limsup_{n\rightarrow\infty}\frac{n_{\alpha}^{\prime}\left(X_{n}(\mathbf{i})\right)}{n_{\alpha}\left(X_{n}(\mathbf{i})\right)}\leq\lim_{n\rightarrow\infty}\frac{\left\lfloor \log_{\alpha}np\right\rfloor +1}{n(1-p)}=0
\]
$\mathbb{P}$-a.e., which shows that the integer part does not contribute
to the asymptotical frequency of digits, i.e. it suffices to analyze
$\tau_{k}^{\alpha}\left(X_{n}-\left\lfloor X_{n}\right\rfloor \right)$. 

Let $Y_{n}=\left(X_{n},i_{n+1}\right)$ and observe that $Y_{n}$
is a Markov chain with state space $\mathbb{X}=[0,\infty)\times\left\{ 1,2\right\} $
and stationary distribution $\pi_{Y}=\pi\times\mathbb{P}$. Set $A=\left\{ (x,2):\,x\in[0,\infty)\right\} $
and let $T^{n}(A)$ denote the $n$:th visit of $Y_{n}$ in $A$.
Since $Y_{n}$ is ergodic, $Z_{n}=Y_{T^{n}(A)}$ is also a Markov
chain, with stationary distribution $\pi_{Z}=\pi_{Y}/\pi_{Y}(A)=\pi_{Y}/(1-p)$.
Now define $h_{k}:\mathbb{X}\rightarrow\left\{ 0,1\right\} $ by 
\[
h_{k}\left(Y_{n}\right)=\begin{cases}
1, & \text{if }\left\lfloor X_{n}\right\rfloor \mod\alpha=k\text{ and }i_{n+1}=2\\
0, & \text{otherwise.}
\end{cases}
\]
Informally, whenever $X_{n+1}$ adds a digit to the $\alpha$-expansion
of $X_{n}-\left\lfloor X_{n}\right\rfloor $, $h_{k}\left(Y_{n+1}\right)$
will equal $1$ if the added digit is $k$. This means that 
\[
\tau_{k}^{\alpha}\left(Z_{1,n}-\left\lfloor Z_{1,n}\right\rfloor ,n\right)=\sum_{i=1}^{n}h_{k}\left(Z_{i}\right),
\]
where $Z_{1,n}$ denotes the first coordinate of $Z_{n}$. While $h_{k}$
is not continuous on $\mathbb{X}$, it is continuous on $\left([0,\infty)\setminus\mathbb{Z}\right)\times\left\{ 1,2\right\} $.
Thus, for any $\epsilon>0$, we can find continuous functions $\overline{h}_{k,\epsilon}$,$\underline{h}_{k,\epsilon}:\mathbb{X}\rightarrow\left[0,1\right]$
such that $\underline{h}_{k,\epsilon}\left(Y_{n}\right)<h_{k}\left(Y_{n}\right)<\overline{h}_{k,\epsilon}\left(Y_{n}\right)$
for all $n\geq0$ and for any $m=0,1,2,\ldots$,
\begin{eqnarray*}
\overline{h}_{k,\epsilon}\left(x,2\right) & = & \begin{cases}
\leq1, & \text{for any }x\in\left[m\alpha+k-\epsilon,m\alpha+k+1+\epsilon\right)\\
0, & \text{otherwise}
\end{cases}\\
\underline{h}_{k,\epsilon}\left(x,2\right) & = & \begin{cases}
1, & \text{for any }x\in\left[m\alpha+k-\epsilon,m\alpha+k+1+\epsilon\right)\\
<1, & \text{otherwise}
\end{cases}
\end{eqnarray*}
Now, by an ergodic theorem of Elton (\cite{Elton1987}), for $f=\overline{h}_{k,\epsilon},\underline{h}_{k,\epsilon}$,
for $\mathbb{P}$-almost every $\mathbf{i}$, 
\begin{equation}
\lim_{n\rightarrow\infty}\frac{1}{n}\sum_{i=1}^{n-1}f\left(Z_{i}(\mathbf{i})\right)=\int fd\pi_{Z},\label{eq:elton}
\end{equation}
for all initial points $Z_{0}\in\mathbb{X}$. Thus, for every $\epsilon>0$,
\begin{eqnarray*}
\limsup_{n\rightarrow\infty}\frac{1}{n}\tau_{k}^{\alpha}\left(Z_{1,n},n\right) & < & \sum_{m=0}^{\infty}\pi\left[m\alpha+k-\epsilon,m\alpha+k+1+\epsilon\right]\\
\liminf_{n\rightarrow\infty}\frac{1}{n}\tau_{k}^{\alpha}\left(Z_{1,n},n\right) & > & \sum_{m=0}^{\infty}\pi\left[m\alpha+k+\epsilon,m\alpha+k+1-\epsilon\right].
\end{eqnarray*}
This means that for $k=0,1,\ldots,\alpha-1$, $\mathbb{P}$-a.e.,
\begin{equation}
\lim_{n\rightarrow\infty}\frac{1}{n}\tau_{k}^{\alpha}\left(Z_{1,n},n\right)=\xi_{k}\label{eq:Zlim}
\end{equation}
where $Z_{1,n}(\mathbf{i})=X_{T^{n}(A)}(\mathbf{i})$. The convergence
is independent of $X_{0}$. Now define the ``backward'' process
$\tilde{X}_{n}(\mathbf{i})=w_{i_{1}}\circ w_{i_{2}}\circ\cdots\circ w_{i_{n}}\left(X_{0}\right)$.
By (\ref{eq:nu}), $\tilde{X}_{n}$ converges $\mathbb{P}$-a.e. to
$\nu\left(\mathbf{i}\right)$, which has distribution $\pi$, since
the distribution of $X_{n}$ (which is the same for $\tilde{X_{n}}$)
converges to $\pi$. Furthermore, (\ref{eq:lln}) must also hold for
$\tilde{X}_{n}$ since $\tilde{X}_{n}$ has the same distribution
as $X_{n}$. As $n_{\alpha}\left(X_{n}(\mathbf{i})\right)\rightarrow\infty$,
(\ref{eq:Zlim}) implies

\[
\frac{\tau_{k}^{\alpha}\left(X_{n}(\mathbf{i}),n_{\alpha}\left(X_{n}(\mathbf{i})\right)\right)}{n_{\alpha}\left(X_{n}(\mathbf{i})\right)}\rightarrow\xi_{k},
\]
$\mathbb{P}$-a.e., and the same claim must again also hold for $\tilde{X_{n}}$.
It follows that $\mathbb{P}$-a.e., $\tau_{k}^{\alpha}\left(\nu(\mathbf{i}),n\right)/n\rightarrow\xi_{k}$,
and the proof is complete.\end{proof}Our main theorem now follows
from the above lemmas:

\begin{proof}[Proof of theorem \ref{main}]Lemma \ref{erg} implies
that $\pi\left(F_{\alpha}\left(\xi_{0},\xi_{1},\ldots,\xi_{\alpha-1}\right)\right)=1$,
so (\ref{eq:mainup}) follows immediately from (\ref{eq:egg}). Now,
assume that $\beta=\alpha^{t}$ for some $t=0,1,\ldots$ Then, for
any $x$, $x/\beta$ will have the same digit expansion as $x$. Thus,
lemmas \ref{erg} and \ref{dimension} together give (\ref{eq:mainlo}).
For the last part, note that for any $x\in S_{\alpha,\overline{d}^{-1}(t)}$,
lemma \ref{dimension} implies $\dim\pi(x)\leq t$ and thus $x\in\overline{E}_{t}$.
An analogous argument shows that $x\in S_{\alpha,\underline{d}^{-1}(t)}$
implies $x\in\underline{E}_{t}$, whereby (\ref{eq:mainmf}) follows.\end{proof}

\begin{remark}\label{packing}If we replace $\liminf$ by $\limsup$
in (\ref{eq:lolocaldim}) and (\ref{eq:updim})-(\ref{eq:lodim}),
we obtain the definitions of the upper and local\emph{ packing} dimensions
of $\mu$, denoted $\dim_{P}\mu$ and $\dim_{P}^{\star}\mu$, respectively.
If $x\in S_{\alpha,y}$ for any $y\in\left[0,1\right]$, the limit
inferior in lemma \ref{dimension} may be dropped in favor of the
ordinary limit. Thus, in the (latter) setting of theorem \ref{main},
\[
\underline{d}\left(\sum_{i=1}^{\alpha-1}i\xi_{i}\right)\leq\dim_{P}\pi\leq\dim_{P}^{\star}\pi\leq\overline{d}\left(\sum_{i=1}^{\alpha-1}i\xi_{i}\right).
\]
\end{remark}

\section{Numerical estimates}

\begin{figure}
\caption{\label{fig:mfs}Lower and upper bounds for $\overline{f}_{H}(t)$,
when $\alpha=5$, $p=1/3$.}

\centering{}\includegraphics[width=0.4\paperwidth]{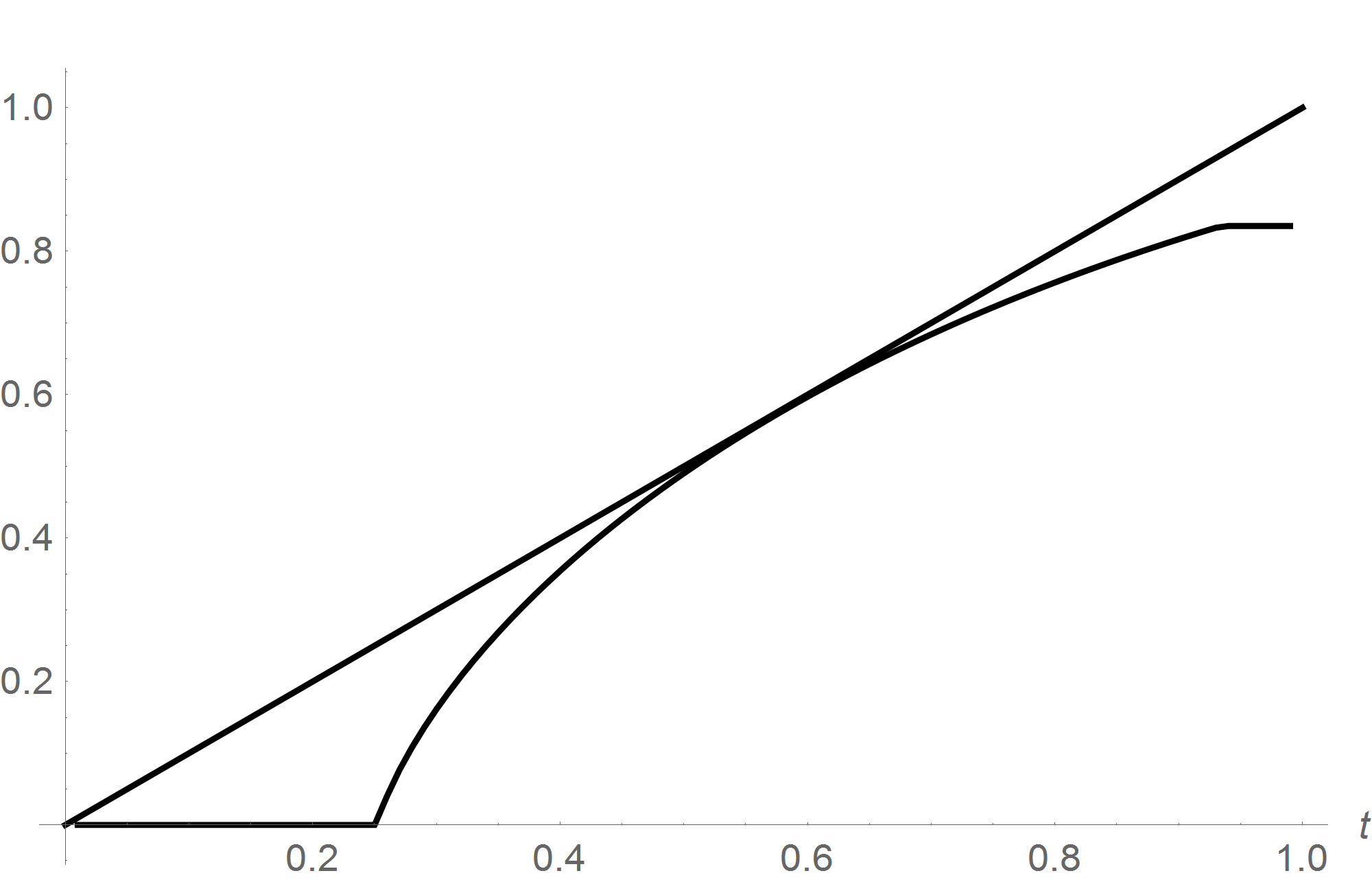}
\end{figure}
When $\beta=1$ we can use the following method to find numerical
approximations of the dimension values in theorem \ref{main}. Since
we only need to evaluate the $\pi$-mass of intervals of unit length,
we partition the state space of $X_{n}$ according to $\left[0,\infty\right)=\bigcup_{i=1}^{\infty}A_{i}$
where $A_{i}=\left[i-1,i\right)$. Now we define a new process $X_{n}^{\prime}$
on $\mathbb{N}$ by the transition probabilities
\[
X_{n+1}^{\prime}=\begin{cases}
X_{n}+1, & \text{with probability }p\\
\left\lfloor X_{n}\alpha^{-1}\right\rfloor , & \text{with probability }1-p
\end{cases}
\]
Note that $X_{n}^{\prime}=m$ whenever $X_{n}\in A_{m}$, since $\alpha$
is an integer and 
\[
\left[m\alpha^{-1},(m+1)\alpha^{-1}\right]\subset\left[\left\lfloor m\alpha^{-1}\right\rfloor ,\left\lfloor m\alpha^{-1}\right\rfloor +1\right]
\]
for all $m\in\mathbb{N}$ (as in lemma \ref{intexp}). The process
$X_{n}^{\prime}$ is called a lumped process of $X_{n}$ (see \cite{Kemeny1976},
section 6.3). Clearly, $X_{n}^{\prime}$ is a Markov process itself,
and it is easily seen that it has stationary distribution $\pi^{\prime}$
defined by $\pi^{\prime}(m)=\pi\left(A_{m}\right)$ for all $m\in\mathbb{N}$. 

We now define the truncated matrix 
\[
P_{n}(i,j)=\begin{cases}
p, & i=j=n\\
P(i,j), & \text{otherwise}
\end{cases}
\]
where the ``missing'' probability is added to the last state to
ensure that the matrix remains stochastic. If we consider the finite
system $\pi_{n}=\pi_{n}P_{n}$, by results of Heyman (\cite{Heyman1991}),
\[
\lim_{n\rightarrow\infty}\pi_{n}(m)=\pi^{\prime}(m)
\]
for all $m\in\mathbb{N}$. This implies that $\lim_{n\rightarrow\infty}\sum_{i=0}^{\infty}\pi_{n}(k+i\alpha+1)=\xi_{k}$,
so by calculating the left eigenvectors of $P_{n}$ for some large
value of $n$ we can find estimates for the dimension of $\pi$ using
theorem \ref{main}. For example, if $\alpha=2$ and $\beta=1$ we
have
\[
P_{5}=\left(\begin{array}{ccccc}
1-p & p & 0 & 0 & 0\\
1-p & 0 & p & 0 & 0\\
0 & 1-p & 0 & p & 0\\
0 & 1-p & 0 & 0 & p\\
0 & 0 & 1-p & 0 & p
\end{array}\right).
\]
Let $p=1/3.$ Now, by calculating the left eigenvectors of $P_{50}$,
we have
\[
0.508\leq\dim_{H}\pi\leq\dim_{H}^{\star}\pi\leq0.906.
\]
The bounds are tighter for larger values of $\alpha$. If we take
$\alpha=5$ instead, we have
\[
0.579\leq\dim_{H}\pi\leq\dim_{H}^{\star}\pi\leq0.585.
\]
In this case, the lower bound to $\overline{f}_{H}(t)$ given by theorem
\ref{main}, along with the upper bound $\overline{f}_{H}(t)\leq t$
(this is standard, see eg. \cite{Falconer1997}) are plotted in figure
(\ref{fig:mfs}). Note that these bounds hold for every $\beta=\alpha^{k}$,
where $k\geq0$ is an integer.

\bibliographystyle{plain}
\bibliography{aku}

\begin{thebibliography}{1}

\bibitem{Barreira2002}
L.~Barreira, B.~Saussol, and J.~Schmeling.
\newblock {Distribution of frequencies of digits via multifractal analysis}.
\newblock {\em Journal of Number Theory}, 97(2):410--438, 2002.

\bibitem{Eggleston1949}
H.~G. Eggleston.
\newblock {The fractional dimension of a set defined by decimal properties}.
\newblock {\em Quarterly Journal of Mathematics}, 20(1):31--36, 1949.

\bibitem{Elton1987}
J~H Elton.
\newblock {An ergodic theorem for iterated maps}.
\newblock {\em Ergodic Theory and Dynamical Systems}, 7:481--488, 1987.

\bibitem{Falconer1997}
Kenneth Falconer.
\newblock {\em {Techniques in Fractal Geometry}}.
\newblock John Wiley \& Sons, 1997.

\bibitem{Freedman1999}
David Freedman and Persi Diaconis.
\newblock {Iterated random functions}.
\newblock {\em SIAM Review}, 41(1):45--76, 1999.

\bibitem{Heyman1991}
Daniel~P. Heyman.
\newblock {Approximating the stationary distribution of an infinite stochastic
  matrix}.
\newblock {\em Journal of Applied Probability}, 28:96--103, 1991.

\bibitem{Jordan2008}
Thomas Jordan and Mark Pollicott.
\newblock {The Hausdorff dimension of measures which contract on average}.
\newblock {\em Discrete and Continuous Dynamical Systems}, 22(9):235--246,
  2008.

\bibitem{Kemeny1976}
John~G. Kemeny and J.~Laurie Snell.
\newblock {\em {Finite Markov Chains}}.
\newblock Springer-Verlag, 1976.

\bibitem{Lau2008}
Ka-Sing Lau, Sze-Man Ngai, and Xiang-Yang Wang.
\newblock {Separation conditions for conformal iterated function systems}.
\newblock {\em Monatshefte f\"{u}r Mathematik}, 156(4):325--355, September
  2008.

\end{thebibliography}

\end{document}